\documentclass[11pt,a4paper]{article} 

\usepackage[ruled,linesnumbered]{algorithm2e} 
\usepackage{amscd} 
\usepackage{amsfonts} 
\usepackage{amsmath} 
\usepackage{amssymb} 
\usepackage{amsthm} 
\usepackage{bm} 
\usepackage{booktabs} 
\usepackage{delarray} 
\usepackage{extarrows} 
\usepackage{geometry} 
\usepackage{graphicx} 
\usepackage{hyperref} 
\usepackage{longtable} 
\usepackage{tikz-cd} 
\usepackage[all]{xy} 
\usepackage[affil-it]{authblk} 

\usepackage{adjustbox}
\usepackage{booktabs} 
\usepackage{caption}
\usepackage{color}
\usepackage{diagbox} 
\usepackage{enumitem} 
\usepackage{float}
\usepackage{hyperref}
\usepackage{indentfirst} 
\usepackage{makecell} 
\usepackage{mathrsfs} 
\usepackage{mathtools} 
\usepackage{multirow} 
\usepackage{relsize} 
\usepackage{subfigure}
\usepackage{quiver}

\setlist[enumerate]{itemsep=0pt,parsep=0pt}

\allowdisplaybreaks[4] 

\def\deg{\operatorname{deg}} 
 
\def\det{\operatorname{det}}

\def\dim{\operatorname{dim}}

\def\Im{\operatorname{Im}}

\def\Mat{\operatorname{Mat}}
\def\max{\operatorname{max}}

\def\mod{\operatorname{mod}}

\def\rad{\operatorname{rad}}
\def\rank{\operatorname{rank}}

\def\sgn{\operatorname{sgn}}

\def\Span{\operatorname{Span}}

\def\tr{\operatorname{tr}}

\def\iff{\Leftrightarrow}

\def\tw{\widetilde{\omega}} 

\def\Eul{\mathrm{Eul}}

\begin{document}

\newtheorem{definition}{Definition}[section] 
\newtheorem{remark}[definition]{Remark}
\newtheorem{example}[definition]{Example}
\newtheorem{proposition}[definition]{Proposition}
\newtheorem{lemma}[definition]{Lemma}
\newtheorem{corollary}[definition]{Corollary}
\newtheorem{theorem}[definition]{Theorem}
\newtheorem{conjecture}[definition]{Conjecture}

\title{Generic Nullity of Generalized Commutators}
\author{Chen Lin$^1$, Chenhao Tang$^2$\footnote{Corresponding author.}, Enhan Zhao$^3$}
\affil{{\small {$^{1}$\textup{School of Mathematics, Nanjing University, Nanjing 210093, China}\\ $^{2}$\textup{Morningside Center of Mathematics, Academy of Mathematics and Systems Science,
Chinese Academy of Sciences; University of the Chinese Academy of Sciences,
Beijing 100190, China}\\$^3$\textup{School of Intelligence Science and Technology, Peking University, Beijing 100871, China}}\\
$^1$\textup{chen.lin@smail.nju.edu.cn},  $^2$\textup{tangchenhao25@mails.ucas.ac.cn}, 
$^3$\textup{morrezhao@stu.pku.edu.cn}} }
\date{}

\maketitle

\begin{abstract}
We study the generic nullity of generalized commutator operators
$$
L_{\mathbf{A}}(X)=s_{k+1}(A_1,\ldots,A_k,X)
$$
on matrix algebras, where $s_{k+1}$ denotes the standard polynomial. Dixon and Pressman conjectured an explicit formula for the generic nullity of $L_{\mathbf{A}}$, and Brassil and Reichstein proved the conjecture when $k$ is even. In this paper, we settle the remaining case where $k$ is odd. Our proof first treats the boundary cases $k=2n-3$ in dimensions $n$ and $n+1$ using degree decompositions and graph-theoretic interpretations of alternating trace forms, and then establishes a dimension-extension argument from $n$ to $n+2$. Consequently, together with the result of Brassil and Reichstein, we obtain a complete proof of the Dixon–Pressman generic nullity conjecture over any field of characteristic zero.

\end{abstract}

\quad\noindent { 2020\it Mathematics Subject Classification:} 15A24, 15A54, 05C38


\tableofcontents

\section{Introduction}
For a positive integer $n$, let $\Mat_n(\mathbb{R})$ denote the algebra of real $n\times n$ matrices and $I_n$ be its identity matrix. We write $S_m$ for the symmetric group on
$\{1,\cdots,m\}$ and $\sgn(\sigma)$ for the sign of a permutation $\sigma$. For $m\ge1$, the standard polynomial of degree $m$ is the multilinear noncommutative polynomial
\[
s_m(X_1,\cdots,X_m)
=\sum_{\sigma\in S_m}\sgn(\sigma)
X_{\sigma(1)}\cdots X_{\sigma(m)}.
\]
It is alternating: interchanging two arguments changes its sign, and hence it vanishes whenever two arguments coincide. For $k\ge1$ and a fixed tuple
$\mathbf{A}=(A_1,\dots,A_k)\in\Mat_n(\mathbb{R})^k$, define the generalized commutator operator
\[
L_{\mathbf{A}}:\Mat_n(\mathbb{R})\longrightarrow\Mat_n(\mathbb{R}),\quad X\longmapsto s_{k+1}(A_1,\dots,A_k,X).
\]
If $k=1$, the kernel of the linear operator $L_{\mathbf{A}}$ is well-understood by basic linear algebra. If $k\ge 2n-1$, the operator $L_{\mathbf{A}}$ is always zero by the Amitsur--Levitzki theorem; see \cite{AmitsurLevitzki,Swan,Swan2}. Thus, the remaining case is $2\le k\le 2n-2$. Dixon and Pressman introduced the problem of determining the generic nullity of $L_{\mathbf{A}}$ in \cite{DixonPressman}. 

\begin{conjecture}[Dixon, Pressman \cite{DixonPressman}]\label{conj}
Let $n\ge2$ and $2\le k\le 2n-2$ be integers. Then for $\mathbf{A}\in\Mat_n(\mathbb{R})^k$ in generic position,
\[
\dim\ker L_{\mathbf{A}}=
 \begin{cases}
  k,&k~\text{even},\\
  k+1,&k~\text{odd},~n~\text{even},\\
  k+2,&k~\text{odd},~n~\text{odd}.
 \end{cases}
\]
\end{conjecture}

In their article, Dixon and Pressman proved the lower bound in all three cases. Brassil and Reichstein proved the conjecture when $k$ is even using a graph-theoretic approach \cite{BrassilReichstein}. The purpose of this paper is to prove the remaining two cases of Conjecture \ref{conj}.

\begin{theorem}\label{thm:main}
Let $n\ge3$, and let $k$ be odd with $3\le k\le 2n-3$. Then there is a nonempty Zariski-open subset
$\mathcal{A}\subseteq\Mat_n(\mathbb{R})^k$ such that, for every $\mathbf{A}\in\mathcal{A}$,
\begin{equation}\label{eq:main-nullity}
 \dim\ker L_{\mathbf{A}}=
 \begin{cases}
  k+1,&n~\text{even},\\
  k+2,&n~\text{odd}.
 \end{cases}
\end{equation}
\end{theorem}

In fact, our proof holds for any field of characteristic zero. Combining Theorem \ref{thm:main} with the result of Brassil and Reichstein \cite{BrassilReichstein}, we obtain a complete answer for Conjecture \ref{conj}.

\begin{theorem}\label{thm:complete}
Let $n\ge2$ and $2\le k\le 2n-2$ be integers. Let $K$ be any field of characteristic zero. Then there is a nonempty Zariski-open subset
$\mathcal{A}\subseteq\Mat_n(K)^k$ such that, for every $\mathbf{A}\in\mathcal{A}$,
\[
\dim\ker L_{\mathbf{A}}=
 \begin{cases}
  k,&k~\text{even},\\
  k+1,&k~\text{odd},~n~\text{even},\\
  k+2,&k~\text{odd},~n~\text{odd}.
 \end{cases}
\]
\end{theorem}

Now we explain the proof strategy of Theorem \ref{thm:main}. Assume $n\ge2$ and $k$ is odd with $3\le k\le 2n-3$. Dixon and Pressman have proved that the dimension in \eqref{eq:main-nullity} is the lower bound, so those tuples $\mathbf{A}$ which make \eqref{eq:main-nullity} hold form a Zariski-open subset of $\Mat_n(\mathbb{R})^k$. We call such an $\mathbf{A}$ a \emph{good} tuple. Thus, it suffices to show a good tuple exists. The proof consists of three steps:

\begin{enumerate}[label=\textup{(\roman*)}]
\item In the boundary case $k=2n-3$, we prove that a good tuple exists for $(k,n)=(2n-3,n)$.
\item We prove that a good tuple exists for $(2n-3,n+1)$.
\item We prove that if a good tuple exists for $(k,n)$, then it also exists for $(k,n+2)$. 
\end{enumerate}
Obviously the three steps imply Theorem \ref{thm:main}.

We briefly describe the main ideas behind these three steps. For the boundary case $k=2n-3$, we impose a grading on $\Mat_n$ by matrix diagonals, under which $L_{\mathbf{A}}$ decomposes into degree blocks. The ranks of the nonzero-degree blocks are controlled by interpreting alternating trace expressions as signed sums of Eulerian circuits. The zero-degree block requires a more subtle argument. To pass from dimension $n$ to $n+1$, we use a block embedding for the nonzero-degree blocks and a further perturbation for the zero-degree block. Finally, the dimension-extension step from $n$ to $n+2$ is obtained by embedding a good tuple into $\Mat_{n+2}$ and then introducing a suitable perturbation to recover the required rank on the remaining $2\times2$ block.

\section{Algebraic and graph-theoretic preliminaries}

\subsection{Alternating trace form}

\begin{lemma}\label{lem:identity-insertion}
Let $m\ge2$, $X_1,\dots,X_{m-1}\in\Mat_n(\mathbb{R})$ and $I_n$ be the identity matrix. Then
\begin{equation}\label{eq:identity-insertion}
s_m(X_1,\dots,X_{m-1},I_n)=
 \begin{cases}
  0,&m~\text{even},\\
  s_{m-1}(X_1,\dots,X_{m-1}),&m~\text{odd}.
 \end{cases}
\end{equation}
\end{lemma}

\begin{proof}
    By definition, the expansion of $s_m(X_1,\dots,X_{m-1},I_n)$ has $m!$ monomials indexed by $\sigma\in S_m$. Group them by the position of $I_n$. Each group contributes a $s_{m-1}(X_1,\dots,X_{m-1})$ with a sign determined by the position of $I_n$. Taking the sum over the groups gives \eqref{eq:identity-insertion}.
\end{proof}

Assume from now on that $k$ is odd. For $\mathbf{A}=(A_1,\cdots,A_k)\in\Mat_n(\mathbb{R})^k$, define a bilinear form on $\Mat_n(\mathbb{R})$ by
\[
w_{\mathbf{A}}(X,Y)=\tr\left(XL_{\mathbf{A}}(Y)\right).
\]
Then we have
\begin{equation}\label{eq:rad-kernel}
\rank w_{\mathbf{A}}=\rank L_{\mathbf{A}}
\end{equation}
because the trace pairing $(X,Z)\mapsto\tr(XZ)$ on $\Mat_n(\mathbb{R})$ is nondegenerate.

\begin{proposition}\label{prop:alternating-trace}
For odd $k$, we have
\begin{equation}\label{eq:cyclic-trace-identity}
 \tr s_{k+2}(A_1,\cdots,A_k,X,Y)=-(k+2)w_{\mathbf{A}}(X,Y).
\end{equation}
Consequently, the bilinear form $w_{\mathbf{A}}$ is alternating.
\end{proposition}

\begin{proof}
The expansion of the left-hand side of
\eqref{eq:cyclic-trace-identity} has $(k+2)!$ monomials indexed by $\sigma\in S_{k+2}$. Group them by the position of $X$. Rotate the unique occurrence of $X$ in each monomial to the first position by the cyclic property of trace. Then each group contributes a $\tr\left(XL_{\mathbf{A}}(Y)\right)$ up to a sign, which is $-1$ by a simple combinatorial argument. Taking the sum over the groups gives \eqref{eq:cyclic-trace-identity}.
\end{proof}

\begin{corollary}\label{nullity-lower-bound}
For $n\ge 3$, odd $k$ with $3\le k\le2n-3$ and $\mathbf{A}=(A_1,\cdots,A_k)\in\Mat_n(\mathbb{R})^k$,
\begin{equation}\label{eq:nullity-lower-bound}
\dim\ker L_{\mathbf{A}}\ge
 \begin{cases}
  k+1,&n~\text{even},\\
  k+2,&n~\text{odd}.
 \end{cases}
\end{equation}
\end{corollary}

\begin{proof}
If $I_n,A_1,\dots,A_k$ are linearly dependent, then by \eqref{eq:identity-insertion}
\[
L_{\mathbf{A}}(X)=s_{k+1}(A_1,\dots,A_k,X)=s_{k+2}(A_1,\dots,A_k,X,I_n)=0
\]
for any $X\in\Mat_n(\mathbb{R})$ because $s_{k+2}$ is multilinear and alternating. If $I_n,A_1,\dots,A_k$ are linearly independent, then $I_n,A_1,\dots,A_k\in\ker L_{\mathbf{A}}$, which gives \eqref{eq:nullity-lower-bound} for $n$ even. By linear algebra, an alternating bilinear form over a field of characteristic different from $2$ has even rank. Therefore by \eqref{eq:rad-kernel}
\[
 \dim\ker L_{\mathbf{A}}=\dim\Mat_n(\mathbb{R})-\rank w_{\mathbf{A}}\equiv n^2\pmod2,
\]
which gives \eqref{eq:nullity-lower-bound} for $n$ odd.
\end{proof}

\begin{definition}\label{def:good-tuple}
We say that a tuple $\mathbf{A}\in\Mat_n(\mathbb{R})^k$ is \emph{good} if equality holds in \eqref{eq:nullity-lower-bound}.
\end{definition}

\subsection{Eulerian circuits and bidirected rooted graphs}

We consider directed graphs with vertices labeled by $\{0,1,\cdots,n-1\}$. We admit repeated edges and loops. An Eulerian path in a directed graph $\Gamma$ is a path visiting every edge exactly once. An Eulerian path is an Eulerian circuit if its initial and terminal vertices coincide. We denote by $\Eul_c(\Gamma)$ the set of Eulerian circuits on $\Gamma$. After labeling the edges $e_1,\cdots,e_m$ in $\Gamma$, we define the sign $\sgn(w)$ of an Eulerian path $w = (e_{\sigma(1)},\cdots,e_{\sigma(m)})$ to be the sign of the permutation $\sigma\in S_m$.

For matrices in $\Mat_n(\mathbb{R})$, we index rows and columns by $0,1,\cdots,n-1$. The matrix unit $E_{u,v}~(0\le u,v\le n-1)$ has a single
$1$ in row $u$ and column $v$. We can regard $E_{u,v}$ as a directed edge
$u\to v$ on $\{0,1,\cdots,n-1\}$. The following lemma is elementary.

\begin{lemma}\label{lem:trace-euler}
Let $E_1,\cdots,E_m$ be matrix units, and let $\Gamma$ be the directed graph with edges $e_i~(1\le i\le m)$ corresponding to $E_i$.
Then for any fixed $1\le j\le m$,
\[
 \tr s_m(E_1,\dots,E_m)
 =\sum_{w\in\Eul_c(\Gamma)}\sgn(w)=
 \begin{cases}
  0,&m~\text{even},\\
  m\sum_{w=(e_j,\cdots)}\sgn(w),&m~\text{odd}.
 \end{cases}
\]
\end{lemma}

\begin{proof}
    The first equality is immediate from the property of the multiplication of matrix units. For the second equality, assume $w = (e_{\sigma(1)},\cdots,e_{\sigma(m)})\in\Eul_c(\Gamma)$ for some $\sigma\in S_m$. For $1\le i\le m$, let $w_i=(e_{\sigma(i)},e_{\sigma(i+1)},\cdots,e_{\sigma(m)},e_{\sigma(1)},\cdots,e_{\sigma(i-1)})\in\Eul_c(\Gamma)$. Then $\sgn(w_i)=-\sgn(w_{i+1})$ if $m$ is even and $\sgn(w_i)=\sgn(w_{i+1})$ if $m$ is odd (here we use the convention $w_{m+1}=w_1$). Taking sum over $1\le i\le m$ gives the result.
    
\end{proof}

\begin{definition}
An undirected graph is \emph{simple} if it has no loops and no repeated edges. A \emph{tree} is a connected simple graph with no circuit. A \emph{forest} is a simple graph whose connected components are trees.
\end{definition}

\begin{definition}\label{def:rooted-forest}
A \emph{rooted forest} $(F,\rho)$ is a forest $F$ together with a distinguished vertex $\rho$, called the root. Its \emph{bidirected rooted graph} $\overleftrightarrow{F}_{\rho}$ is obtained by replacing each edge
$\{u,v\}$ with the two directed edges $u\to v$ and $v\to u$ and adjoining one additional loop $\rho\to\rho$ at the root.
\end{definition}

For a vertex $v$ of an undirected graph $\Gamma$, write $\deg_{\Gamma}(v)$ for the number of edges containing $v$. The following lemma will be used in Subsection \ref{subsec:nonzero-block}.

\begin{lemma}\label{lem:bidirected-tree}
Let $(T,\rho)$ be a rooted tree on $\{0,1,\cdots,n-1\}$ and $\overleftrightarrow{T}_{\rho}$ be its associated bidirected rooted graph. After labeling the $2n-1$ edges of $\overleftrightarrow{T}_{\rho}$, we have
\begin{equation}\label{eq:bidirected-tree-coefficient}
\left|\sum_{w\in\Eul_c(\overleftrightarrow{T}_{\rho})}\sgn(w)\right|
=(2n-1)\deg_T(\rho)!\prod_{v\ne\rho}(\deg_T(v)-1)!\ne0.
\end{equation}
\end{lemma}

\begin{proof}
By Lemma \ref{lem:trace-euler}, it suffices to take the sum over all Eulerian circuits starting from the root loop, and then multiply by the number of edges $2n-1$. A similar argument of \cite[Lemma 2.5]{BrassilReichstein} gives that all such Eulerian circuits have the same sign. The multiplication principle gives the count $\deg_T(\rho)!\prod_{v\ne\rho}(\deg_T(v)-1)!$ and hence \eqref{eq:bidirected-tree-coefficient}.
\end{proof}

The following lemma will be used in Subsection \ref{subsec:zero-block}.

\begin{lemma}\label{lem:odd-circuit}
Assume $n\ge 4$. Let $\Gamma$ be a directed graph on $\{0,1,...,n-1\}$ with edges $0\to0,~0\to1,~1\to n-1,~n-1\to2,~2\to0$ and $0\to i,~i\to0$ for $3\le i\le n-2$. Choose $j_1,j_2\in \{0,1,\cdots,n-2\}$ and let $\Gamma'$ be the directed graph obtained by adding two loops $j_1\to j_1,~j_2\to j_2$ to $\Gamma$. Then
\[
\sum_{w\in\Eul_c(\Gamma')}\sgn(w)\ne0\iff\{j_1,j_2\}=\{1,2\},
\]
in which case we have
\begin{equation}\label{eq:circuit-counting}
\left|\sum_{w\in\Eul_c(\Gamma')}\sgn(w)\right|
=(2n-1)(n-3)!\ne0.
\end{equation}
\end{lemma}

\begin{proof}
    By Lemma \ref{lem:trace-euler}, it suffices to take the sum over all Eulerian circuits starting from the loop $0\to0$, and then multiply by the number of edges $2n-1$. By \cite[Lemma 2.2]{BrassilReichstein}, the sum of the signs of all Eulerian circuits of a graph with repeated edges is zero. Thus, we can assume $j_1\ne j_2$ and $j_1,j_2\ne0$. Note that any Eulerian circuit $w$ of $\Gamma'$ (starting from $0\to0$) is a concatenation of circuits from $0$ to $0$. If $\{j_1,j_2\}\ne\{1,2\}$, then $w$ contains exactly two such circuits whose numbers of edges are odd (the loop $0\to0$ is excluded here), and if we exchange the two circuits, we will get a new Eulerian circuit with opposite sign, which makes the entire sum zero.

    If $\{j_1,j_2\}=\{1,2\}$, then the number of edges of any circuit from $0$ to $0$, except the loop $0\to0$, is even. Thus, all Eulerian circuits of $\Gamma'$ (starting from $0\to0$) have the same sign. The multiplication principle gives the count $(n-3)!$ and hence \eqref{eq:circuit-counting}. 
\end{proof}

\section{The boundary case \texorpdfstring{$k=2n-3$}{k=2n-3}}
\label{sec:boundary}

Throughout this section, assume that $n\ge 3$ and $k=2n-3$. The goal of this section is to prove the following theorem.

\begin{theorem}\label{thm:boundary}
In the case $k=2n-3$, a good tuple exists.  
\end{theorem}

\subsection{Degree blocks and perfect trace pairings}

For matrices in $\Mat_{n}(\mathbb{R})$, we index rows and columns by $0,1,\cdots,n-1$. For $-(n-1)\le q\le n-1$, set the degree-$q$ part of $\Mat_{n}(\mathbb{R})$
\[
\mathcal{M}_q=\Span\{E_{a,a+q}\mid 0\le a,a+q\le n-1\}\subset\Mat_{n}(\mathbb{R}).
\]
Then $\dim\mathcal{M}_q=n-|q|$ and we have a decomposition
\[
\Mat_{n}(\mathbb{R})=\bigoplus_{q=-(n-1)}^{n-1}\mathcal M_q.
\]

Let $\mathbf{A}=(A_{-(n-2)},\cdots,A_{n-2})$ be a tuple with $A_i\in\mathcal{M}_i,~-(n-2)\le i\le n-2$. As degrees add under matrix multiplication, the operator $L_{\mathbf{A}}$ preserves each $\mathcal{M}_q$, and hence it decomposes into degree blocks
\[
 L_{\mathbf{A}}=\bigoplus_{q=-(n-1)}^{n-1}L_{\mathbf{A},q} 
\]
where $L_{\mathbf{A},q}=L_{\mathbf{A}}|_{\mathcal{M}_q}$. For each $q$, define
\[
w_{\mathbf{A},q}:\mathcal{M}_{-q}\times\mathcal{M}_q\longrightarrow\mathbb{R},\quad(X,Y)\longmapsto\tr\left(XL_{\mathbf{A},q}(Y)\right).
\]

\begin{lemma}\label{lem:block-trace-rank}
For every $-(n-1)\le q\le n-1$,
\begin{equation}\label{eq:block-rank-equals-pairing-rank}
\rank w_{\mathbf{A},q}=\rank L_{\mathbf{A},q}.
\end{equation}
For $0<q\le n-1$,
\begin{equation}\label{eq:opposite-block-ranks}
\rank L_{\mathbf{A},-q}=\rank L_{\mathbf{A},q}.
\end{equation}
\end{lemma}

\begin{proof}
For each $q$, the trace pairing restricts to a perfect pairing $\mathcal{M}_{-q}\times \mathcal{M}_q\rightarrow\mathbb{R}$. This gives
\eqref{eq:block-rank-equals-pairing-rank}. For $q>0,~X\in\mathcal{M}_{-q},~Y\in\mathcal{M}_q$, by Proposition \ref{prop:alternating-trace} $ w_{\mathbf{A},q}(X,Y)=-w_{\mathbf{A},-q}(Y,X)$. This gives \eqref{eq:opposite-block-ranks}.
\end{proof}

\begin{lemma}\label{lem:block-nullity-lower-bound}
For every $-(n-1)\le q\le n-1$,
\begin{equation}\label{eq:block-nullity-lower-bound}
\dim\ker L_{\mathbf{A},q}\ge
\begin{cases}
  0,&q=\pm(n-1)\\
  1,&q\ne 0,\pm(n-1)\\
  2,&q=0,~n~\text{even},\\
  3,&q=0,~n~\text{odd}.
\end{cases}
\end{equation}    
\end{lemma}

\begin{proof}
The proof is similar to that of Corollary \ref{nullity-lower-bound}.
\end{proof}

Put $\mathcal{H}=\prod_{i=-(n-2)}^{n-2}\mathcal{M}_i$, which is an irreducible real affine space. By Lemma \ref{lem:block-nullity-lower-bound}, for fixed $q$, those tuples $\mathbf{A}\in\mathcal{H}$ which make equality hold in \eqref{eq:block-nullity-lower-bound} form a Zariski-open subset of $\mathcal{H}$. We will show the subset is nonempty. 

\subsection{The nonzero-degree blocks}
\label{subsec:nonzero-block}

In this subsection, for $q\ne0$ and $-(n-1)\le q\le n-1$, we will construct a concrete tuple $\mathbf{A}\in\mathcal{H}$ which makes equality hold in \eqref{eq:block-nullity-lower-bound}. By \eqref{eq:opposite-block-ranks} it suffices to deal with $1\le q\le n-1$. The construction is graph-theoretic.

For an undirected edge $\{u,v\}$ on $\{0,\cdots,n-1\}$, call $|u-v|$ its
\emph{length}.

\begin{lemma}\label{lem:q-forest}
Fix $1\le q<n-1$. There exists a forest $F$ on
$\{0,1,\cdots,n-1\}$, satisfying
\begin{enumerate}
\item $F$ has two connected components;
\item for each $m\in\{1,\cdots,n-2\}$, exactly one edge of $F$ has length $m$;
\item the edge of length $q$ is $\{0,q\}$, and among all pairs $\{a,a+q\}$ with $0\le a\le n-1-q$, this is the unique pair lying in the same component of $F$.
\end{enumerate}
Call such a forest a $q$-forest.
\end{lemma}

The proof of Lemma \ref{lem:q-forest} is a little tedious, and we present it in the next subsection.
\begin{example} 
Let $(n,q)=(8,3)$. Take the two groups
\[
C_0=\{0,1,3,5,7\},\quad C_1=\{2,4,6\},
\]
and the edges
\[
\{0,1\},~\{4,6\},~\{0,3\},~\{2,6\},~\{0,5\},~\{1,7\}.
\]
Their lengths are respectively $1,2,3,4,5,6$. Among the pairs $\{a,a+3\}$, only $\{0,3\}$ is internal to a group. Therefore, this is a $3$-forest.
\end{example}

For $1\le q<n-1$, choose a $q$-forest $F$ and write its edge of length $m~(1\le m\le n-2)$ as $\{u_m,v_m\}$ with $u_m<v_m$. Specialize the tuple $\mathbf{A}\in\mathcal{H}$ by
\begin{equation}\label{eq:configuration-specialization}
A_m=E_{u_m,v_m},\quad A_{-m}=E_{v_m,u_m},\quad A_0=E_{0,0}.
\end{equation}
For $0\le a,b\le n-q-1$, put
\[
Y_a=E_{a,a+q}\in\mathcal{M}_q,\quad X_b=E_{b+q,b}\in\mathcal{M}_{-q}.
\]

\begin{proposition}\label{prop:nonzero-block}
For $1\le q< n-1$ and $\mathbf{A}$ in \eqref{eq:configuration-specialization}, the matrix
$\left(w_{\mathbf{A},q}(X_b,Y_a)\right)_{0\le a,b\le n-q-1}$ is diagonal. Its entry indexed
by $a=b=0$ is zero and every other diagonal entry is nonzero. In particular,
\[
\rank L_{\mathbf{A},q}=n-q-1,\quad\dim\ker L_{\mathbf{A},q}=1.
\]
\end{proposition}

\begin{proof}
Let $\Gamma=\overleftrightarrow{F}_0\cup\{y_a:a\to a+q\}\cup\{x_b:b+q\to b\}$ (see Definition \ref{def:rooted-forest} for the notation $\overleftrightarrow{F}_0$). By \eqref{eq:cyclic-trace-identity}, up to a nonzero constant, the entry $w_{\mathbf{A},q}(X_b,Y_a)$ is given by $\tr s_{2n-1}(\mathbf{A},X_b,Y_a)$, which by Lemma \ref{lem:trace-euler} can be converted to the sum of the signs of all Eulerian circuits of the directed graph $\Gamma$. If this entry does not vanish, then $\Gamma$ must admit an Eulerian circuit. In particular, the outdegree must be equal to the indegree for every vertex. For the bidirected rooted graph $\overleftrightarrow{F}_0$, by construction the outdegree is equal to the indegree for every vertex. Therefore, one must have $a=b$. Hence the pairing matrix is diagonal.

For $a=b=0$, the directed graph $\Gamma$ has repeated edges $0\to q$ (one from $y_a$ and another from the $q$-forest $F$), so the entry $a=b=0$ is zero by \cite[Lemma 2.2]{BrassilReichstein}. For $a=b\ne0$, let $T$ be the graph obtained from $F$ by adding the edge $\{a,a+q\}$. Then by the definition of $q$-forest $T$ is a tree and $\Gamma=\overleftrightarrow{T}_0$. Therefore, the entry $a=b\ne0$ is nonzero by Lemma \ref{lem:bidirected-tree}.

\end{proof}

For $q=n-1$, set
\[
A_0=E_{0,0},\quad A_m=E_{0,m},\quad A_{-m}=E_{m,0}\quad (1\le m\le n-2)
\]
and $X=E_{n-1,0},~Y=E_{0,n-1}$. Let $T$ be the tree on $\{0,1,\cdots, n-1\}$ with edges $\{\{0,m\}\mid 1\le m\le n-1\}$. Then up to a nonzero constant, the value $w_{\mathbf{A},q}(X,Y)$ is equal to the sum of the signs of all Eulerian circuits of the directed graph $\overleftrightarrow{T}_0$, which is nonzero again by Lemma \ref{lem:bidirected-tree}.
Therefore
\[
\rank L_{\mathbf{A},n-1}=1,\quad\dim\ker L_{\mathbf{A},n-1}=0.
\]

\subsection{Existence of \texorpdfstring{$q$}{q}-forests}
\label{subsec}

In this subsection, we prove Lemma \ref{lem:q-forest}.

\begin{lemma}\label{lem:valuation}
Assume $1\le q<n-1$. Set
\[
M=
\begin{cases}
n-1,& n-1-q\text{ is odd},\\
n-2,& n-1-q\text{ is even}.
\end{cases}
\]
Then one can assign values $c_0,\cdots,c_M\in\{0,1\}$ such that
\begin{equation}\label{eq:boundary-valuation}
c_0=c_q=0,\quad c_M=1
\end{equation}
and
\begin{equation}\label{eq:translation-symmetry}
c_k=c_{M-k}\quad(1\le k<M),\quad c_{l+q}=1-c_l\quad(1\le l\le M-q).
\end{equation}
Moreover, if $q\ge3$, one can further require $c_1\ne c_{q-1}$.
\end{lemma}

\begin{proof}
For $a\in\mathbb{Z}/q\mathbb{Z}$, let $L_a=\{1\le x\le M-1\mid x\equiv a\pmod q\}$. Then a valuation on $L_a \cup L_{M-a}$ satisfying \eqref{eq:translation-symmetry} (if it exists) is uniquely determined by the value $c_x$ for any $x\in L_a \cup L_{M-a}$. If $L_a\ne L_{M-a}$, then it is easy to see such a valuation exists. If $L_a=L_{M-a}$, then the only obstruction to such a valuation is $b+q=M-b$ for some $b\in L_a$, which is impossible because $M-q$ is always odd. Thus, a valuation on $L_a \cup L_{M-a}$ satisfying \eqref{eq:translation-symmetry} always exists.

Put $c_q=c_{M-q}=0$ and extend it uniquely to a valuation on $L_q \cup L_{M-q}$ by \eqref{eq:translation-symmetry} and the preceding discussion. For other $L_a \cup L_{M-a}$, choose one of the two possible valuations on it satisfying \eqref{eq:translation-symmetry}. Put $c_0=0,~c_M=1$. We get a valuation for $\{0,1,\cdots,M\}$ satisfying \eqref{eq:boundary-valuation} and \eqref{eq:translation-symmetry}.

Finally, assume $q\ge 3$. If $L_1 \cup L_{M-1}\ne L_{q-1} \cup L_{M-q+1}$, it is easy to see we can require $c_1\ne c_{q-1}$. If $L_1 \cup L_{M-1}= L_{q-1} \cup L_{M-q+1}$, then $q\mid M$. Using \eqref{eq:translation-symmetry}, we have
\[
c_{q-1}=c_{M-q+1}=1-c_1\ne c_1
\]
because $M-q$ is always odd.
\end{proof}

\begin{proof}[Proof of Lemma \ref{lem:q-forest}]
Let $M$ and $c_i ~(0\le i\le M)$ be as in Lemma \ref{lem:valuation}. For $1\le m\le n-2$, we denote by $e_m$ the edge with length $m$. For $1\le m<M$, define
\[
e_m=
\begin{cases}
\{0,m\},&c_m=0,\\
\{M-m,M\},&c_m=1.
\end{cases}
\]
This forms two connected components
\[
C_0=\{0\le x\le M\mid c_x=0\},\quad C_1=\{0\le x\le M\mid c_x=1\}.
\]
Moreover, by \eqref{eq:boundary-valuation} and \eqref{eq:translation-symmetry}, we have $\{0,q\}\subseteq C_0$ and $\{a,a+q\}$ is not in the same connected component for every $1\le a\le M-q$.

If $M=n-1$, this already gives the desired $q$-forest. Suppose that $M=n-2$ and $q\ne 2$. Add the last vertex $n-1$ to the connected component containing $1$ and define $e_{n-2}=\{1,n-1\}$. If $q\ge3$, then $c_{n-1-q}=c_{q-1}\ne c_1$ by Lemma \ref{lem:valuation}. Therefore, $\{n-1-q,n-1\}$ is not in the same connected component and this gives the desired $q$-forest. If $q=1$, then $c_{n-2}=c_M=1,~c_1=c_q=0$. Therefore, $\{n-2,n-1\}$ is not in the same connected component and this again gives the desired $q$-forest.

It remains to treat $q=2$ with $n$ odd, which requires an additional construction. Write either
\[
n-1=4t+2
\quad\text{or}\quad
n-1=4t+4.
\]
If $n-1=4t+2$, set
\[
C_0=\{0\}\cup
\{v\mid 1\le v\le n-1,~v\equiv1,2\pmod4\},
\]
\[
C_1=\{v\mid 1\le v\le n-1,~v\equiv0,3\pmod4\},
\]
and for $1\le m\le n-2=4t+1$, define
\[
e_m=
\begin{cases}
\{0,m\},&m\equiv2\pmod4,\\
\{2,m+2\},&m\equiv3\pmod4,\\
\{3,m+3\},&m<4t,~m\equiv0,1\pmod4,\\
\{1,4t+1\},&m=4t,\\
\{1,4t+2\},&m=4t+1.
\end{cases}
\]
If $n-1=4t+4$, set
\[
C_0=\{0\}\cup
\{v\mid 1\le v\le n-1,~v\equiv2,3\pmod4\},
\]
\[
C_1=\{v\mid 1\le v\le n-1,~v\equiv0,1\pmod4\},
\]
and for $1\le m\le n-2=4t+3$, define
\[
e_m=
\begin{cases}
\{2,m+2\},&m\equiv1\pmod4,\\
\{0,m\},&m\equiv2\pmod4,\\
\{1,m+1\},&m\equiv0,3\pmod4.
\end{cases}
\]
One checks that this gives the desired $q$-forest and completes the proof.
\end{proof}

\subsection{The zero-degree block}
\label{subsec:zero-block}

In this subsection, for $q=0$, we will construct a concrete tuple $\mathbf{A}\in\mathcal{H}$ which makes equality hold in \eqref{eq:block-nullity-lower-bound}. We can assume $n\ge4$.

Set $s=\left\lfloor\frac{n-2}{2}\right\rfloor>0$ and a reflection $r(l)=n-1-l$. We specialize the tuple $\mathbf{A}\in\mathcal{H}$ to a family $\mathbf{A}(t)$ parametrized by $t\in\mathbb{R}$. Set $A_0=E_{0,0}$ and for $1\le l\le s$,
\begin{equation}\label{eq:zero-block-family}
A_l=E_{0,l},~A_{-l}=E_{l,0}+E_{n-1,r(l)},~A_{r(l)}(t)=E_{0,r(l)}+tE_{l,n-1},~A_{-r(l)}=E_{r(l),0}.
\end{equation}
If $n$ is odd, also set
\begin{equation}\label{eq:zero-block-middle-pair}
A_{\frac{n-1}{2}}=E_{0,\frac{n-1}{2}},\quad A_{-\frac{n-1}{2}}=E_{\frac{n-1}{2},0}.
\end{equation}
For $0\le i\le n-1$, set $Z_i=E_{i,i}$.

\begin{proposition}\label{prop:zero-block-first-order}
For $1\le i,j\le s$, we have
\[
    w_{\mathbf{A}(t),0}(Z_i,Z_j)\equiv w_{\mathbf{A}(t),0}(Z_{r(i)},Z_{r(j)})\equiv0~(\mod t^2).
\]
For $1\le i,j\le s,~i\ne j$, we have
\[
    w_{\mathbf{A}(t),0}(Z_i,Z_{r(j)})\equiv w_{\mathbf{A}(t),0}(Z_{r(i)},Z_j)\equiv0~(\mod t^2).
\]
For $1\le i\le s$, there exists a nonzero constant $c_i$ such that
\[
     w_{\mathbf{A}(t),0}(Z_i,Z_{r(i)})=-w_{\mathbf{A}(t),0}(Z_{r(i)},Z_i)\equiv c_i t~(\mod t^2).
\]
\end{proposition}

\begin{proof}
By \eqref{eq:cyclic-trace-identity} and Lemma \ref{lem:trace-euler}, up to a nonzero constant,  we can convert $w_{\mathbf{A}(t),0}(Z_i,Z_j)~(i,j\in[1,s]\cup[r(s),r(1)])$ to the sum of the signs of Eulerian circuits of some directed graphs. In fact, we can expand $w_{\mathbf{A}(t),0}(Z_i,Z_j)$ into $2^{2s}$ terms by decomposing $A_{-l},A_{r(l)}$ in \eqref{eq:zero-block-family} into matrix units. Each term corresponds to a directed graph on $\{0,1,\cdots,n-1\}$ with edges \[
A_0:0\to0,\quad Z_i:i\to i,\quad Z_j:j\to j
\] 
and 
\[
A_l:0\to l,\quad A_{-l}:l\to0~\text{or}~n-1\to r(l)
\]
\[
A_{r(l)}(t):0\to r(l)~\text{or}~l\to n-1,\quad A_{-r(l)}:r(l)\to0
\]
for $1\le l\le s$ (if $n$ is odd, we also include $0\to\frac{n-1}{2}$ and $\frac{n-1}{2}\to0$). 

If this term is nonzero, then the corresponding graph must admit an Eulerian circuit. In particular, the outdegree must be equal to the indegree for every vertex. For every $1\le l\le s$, if the edge corresponding to $A_{-l}$ is $l\to0$, then the edge corresponding to $A_{r(l)}(t)$ must be $0\to r(l)$: otherwise the outdegree of $r(l)$ is larger than the indegree. In this case, the three vertices $0,l,r(l)$ look like
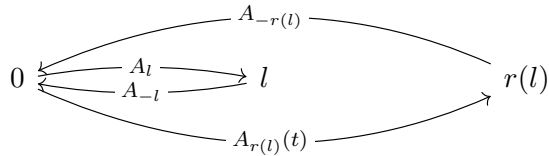
\[\begin{tikzcd}
	0 &&& l &&& {r(l)}
	\arrow["{A_l}"{description}, curve={height=-6pt}, from=1-1, to=1-4]
	\arrow["{A_{r(l)}(t)}"{description}, curve={height=30pt}, from=1-1, to=1-7]
	\arrow["{A_{-l}}"{description}, curve={height=-6pt}, from=1-4, to=1-1]
	\arrow["{A_{-r(l)}}"{description}, curve={height=30pt}, from=1-7, to=1-1]
\end{tikzcd}\]

Similarly, if the edge corresponding to $A_{-l}$ is $n-1\to r(l)$, then the edge corresponding to $A_{r(l)}(t)$ must be $l\to n-1$: otherwise the indegree of $l$ is larger than the outdegree. In this case, the four vertices $0,l,r(l),n-1$ look like
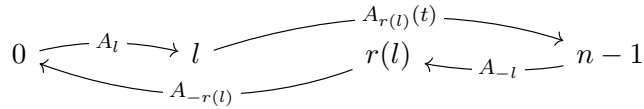
\[\begin{tikzcd}
	0 && l && {r(l)} && {n-1}
	\arrow["{A_l}"{description}, curve={height=-6pt}, from=1-1, to=1-3]
	\arrow["{A_{r(l)}(t)}"{description}, curve={height=-18pt}, from=1-3, to=1-7]
	\arrow["{A_{-r(l)}}"{description}, curve={height=-18pt}, from=1-5, to=1-1]
	\arrow["{A_{-l}}"{description}, curve={height=-6pt}, from=1-7, to=1-5]
\end{tikzcd}\]
and it will contribute a factor $t$ (from $A_{r(l)}(t)$) to the original term. 

Therefore, the graph corresponding to the constant term of $w_{\mathbf{A}(t),0}(Z_i,Z_j)\in\mathbb{Z}[t]$ is the one with edges $0\to0,~i\to i,~j\to j$ and $0\to l,~l\to0$ for $1\le l\le n-2$. No edge contains the vertex $n-1$ and we can delete it. Then we get a directed graph $\Gamma$ with $n-1$ vertices and $2n-1$ edges. By \cite[Theorem 2]{Swan}, we have
\[
\sum_{w\in\Eul_c(\Gamma)}\sgn(w)=0
\]
(one can also check this directly, using a similar argument of Lemma \ref{lem:odd-circuit}). Hence, the constant term of $w_{\mathbf{A}(t),0}(Z_i,Z_j)$ vanishes. 

Similarly, the graphs corresponding to the linear term of $w_{\mathbf{A}(t),0}(Z_i,Z_j)$ are those with edges $0\to0,~i\to i,~j\to j$, $0\to l,~l\to n-1,~n-1\to r(l),~r(l)\to0$ for some $1\le l\le s$ and $0\to l',~l'\to0$ for $1\le l'\le n-2,~l'\ne l,r(l)$. By Lemma \ref{lem:odd-circuit}, the linear term of $w_{\mathbf{A}(t),0}(Z_i,Z_j)$ is nonzero if and only if $\{i,j\}=\{l,r(l)\}$ for some $1\le l\le s$. This finishes the proof.
\end{proof}

Let $W=\Span\{Z_1,\cdots,Z_s,Z_{r(1)},\cdots,Z_{r(s)}\}\subseteq\mathcal{M}_0$. Let $M(t)$ be the matrix of $w_{\mathbf{A}(t),0}|_W$ under the ordered basis $Z_1,\cdots,Z_s,Z_{r(1)},\cdots,Z_{r(s)}$. By Proposition \ref{prop:zero-block-first-order}, we have
\[
M(t)\equiv t
\begin{bmatrix}
0&D\\
-D&0
\end{bmatrix}
(\mod t^2),\quad D=\operatorname{diag}\{c_1,\cdots,c_s\}.
\]
Therefore
\[
\det M(t)\equiv ct^{2s}(\mod t^{2s+1})
\]
for some nonzero constant $c$. In particular, $\det M(t)$ is a nonzero polynomial in $\mathbb{R}[t]$, so there exists $t_0\in \mathbb{R}$ such that
\[
\det M(t_0)\ne0.
\]
Set $\mathbf{A}=\mathbf{A}(t_0)$. For this specialization, we have
\[
\rank L_{\mathbf{A},0}=\rank w_{\mathbf{A},0}\ge\rank w_{\mathbf{A},0}|_W=2s.
\]
Combining this with Lemma \ref{lem:block-nullity-lower-bound}, we obtain
\[
\dim\ker L_{\mathbf{A},0}=
\begin{cases}
2,&n~\text{even},\\
3,&n~\text{odd}.
\end{cases}
\]

\subsection{Proof of Theorem \ref{thm:boundary}}\label{subsec:proof-of-thm-boudary}

\begin{proof}[Proof of Theorem \ref{thm:boundary}]
For fixed $-(n-1)\le q\le n-1$, by Lemma \ref{lem:block-nullity-lower-bound}, those tuples $\mathbf{A}\in\mathcal{H}=\prod_{i=-(n-2)}^{n-2}\mathcal{M}_i$ which make equality hold in \eqref{eq:block-nullity-lower-bound} form a Zariski-open subset of $\mathcal{H}$. By Subsections \ref{subsec:nonzero-block} and \ref{subsec:zero-block}, the subset is nonempty. Since $\mathcal{H}$ is an irreducible real affine space, the finite intersection of nonempty Zariski-open subsets of $\mathcal{H}$ is still nonempty. This implies the existence of a tuple $\mathbf{A}\in\mathcal{H}$ which makes equality hold in \eqref{eq:block-nullity-lower-bound} for every $-(n-1)\le q\le n-1$. For such $\mathbf{A}$, we have
\[
\dim\ker L_{\mathbf{A}}=\sum_{q=-(n-1)}^{n-1}\dim\ker L_{\mathbf{A},q}=
 \begin{cases}
  2n-2,&n~\text{even},\\
  2n-1,&n~\text{odd}.
 \end{cases}
\]
Therefore, such $\mathbf{A}$ is a good tuple, and Theorem \ref{thm:boundary} holds.
\end{proof}

\section{The case \texorpdfstring{$(2n-3,n+1)$}{(2n-3,n+1)}}\label{sec:additional-case}

Throughout this section, assume that $n\ge 3$ and $k=2n-3$. By Theorem \ref{thm:boundary}, a good tuple exists for $(k,n)=(2n-3,n)$. The goal of this section is to prove the following theorem.

\begin{theorem}\label{thm:additional-case}
A good tuple exists for $(2n-3,n+1)$.
\end{theorem}

\subsection{Degree blocks in dimension \texorpdfstring{$n+1$}{n+1}}

For matrices in $\Mat_{n+1}(\mathbb{R})$, we index rows and columns by $0,1,\cdots,n$. For $-n\le q\le n$, set
\[
\mathcal{M}_q'=\Span\{E_{a,a+q}\mid 0\le a,a+q\le n\}
\subset\Mat_{n+1}(\mathbb{R}).
\]
Then $\dim \mathcal{M}_q'=n+1-|q|$ and $\Mat_{n+1}(\mathbb{R})=\bigoplus_{q=-n}^{n}\mathcal{M}_q'$.
Put $\mathcal{H}'=\prod_{i=-(n-2)}^{n-2}\mathcal{M}_i'$.
For a tuple $\mathbf{A}'=(A_{-(n-2)}',\cdots,A_{n-2}')\in\mathcal{H}'$, we have
\[
L_{\mathbf{A}'}=\bigoplus_{q=-n}^{n}L_{\mathbf{A}',q},\quad
L_{\mathbf{A}',q}=L_{\mathbf{A}'}|_{\mathcal{M}_q'}.
\]
For $-n\le q\le n$, define
\[
w_{\mathbf{A}',q}:\mathcal{M}_{-q}'\times \mathcal{M}_q'\longrightarrow\mathbb{R},\quad (X,Y)\longmapsto\tr\left(XL_{\mathbf{A}',q}(Y)\right).
\]
The same argument of Lemma \ref{lem:block-trace-rank}, \ref{lem:block-nullity-lower-bound} gives
\begin{equation}\label{eq:additional-opposite-block-ranks}
    \rank w_{\mathbf{A}',q}=\rank L_{\mathbf{A}',q},\quad\rank L_{\mathbf{A}',-q}=\rank L_{\mathbf{A}',q}
\end{equation}
and
\begin{equation}\label{eq:block-nullity-lower-bound-additional}
    \dim\ker L_{\mathbf{A}',q}\ge
\begin{cases}
  0,&q=\pm(n-1),\pm n;\\
  1,&q\ne 0,\pm(n-1),\pm n;\\
  2,&q=0,~n~\text{odd};\\
  3,&q=0,~n~\text{even}.
\end{cases}
\end{equation}
For fixed $q$, those tuples $\mathbf{A}'\in\mathcal{H}'$ which make equality hold in \eqref{eq:block-nullity-lower-bound-additional} form a Zariski-open subset of $\mathcal{H}'$. We will show the subset is nonempty. 

\subsection{The nonzero-degree blocks}

Let $\mathcal{H}=\prod_{i=-(n-2)}^{n-2}\mathcal{M}_i$ be the real affine space defined in Section \ref{sec:boundary}. For a tuple $\mathbf{A}=(A_{-(n-2)},\cdots,A_{n-2})\in \mathcal{H}$, we have $s_{2n-3}(\mathbf{A})\in\mathcal{M}_0$. Write
\[
s_{2n-3}(\mathbf{A})=\operatorname{diag}
\{
g_0(\mathbf{A}),\cdots,g_{n-1}(\mathbf{A})
\}
\]
for polynomial functions $g_i$ on $\mathcal{H}$.

\begin{lemma}\label{lem:g_a-nonzero}
For every $0\le i\le n-1$, the polynomial $g_i$ is
nonzero.
\end{lemma}

\begin{proof}
First assume that $0\le i\le n-2$. Specialize $A_0=E_{0,0}$ and $A_j=E_{0,j},~A_{-j}=E_{j,0}$ for $1\le j\le n-2$. The edges of the corresponding directed graph $\Gamma$ are $0\to j,~j\to0$ for $1\le j\le n-2$ and the loop $0\to0$. The vertex $n-1$ is isolated. By the property of the multiplication of matrix units, up to a sign $g_i(\mathbf{A})$ is equal to the sum of the signs of all Eulerian circuits of $\Gamma$ starting from the vertex $i$, which is nonzero by Lemma \ref{lem:trace-euler} and \ref{lem:bidirected-tree}.

For $i=n-1$, specialize $A_0=E_{n-1,n-1}$ and $A_j=E_{n-1-j,n-1},~A_{-j}=E_{n-1,n-1-j}$ for $1\le j\le n-2$.
Now the edges of the corresponding directed graph are $n-1\to j,~j\to n-1$ for $1\le j\le n-2$ and the loop $n-1\to n-1$. The vertex $0$ is isolated. The same argument
gives $g_{n-1}(\mathbf{A})\ne0$.
\end{proof}

For $A_i\in\mathcal{M}_i$, set 
\begin{equation}\label{eq:nonzero-diagonal-embed}
A_i'=\operatorname{diag}\{A_i,0\}
\in\mathcal{M}_i'
\end{equation}
and $\mathbf{A}'=(A_{-(n-2)}',\cdots,A_{n-2}')\in\mathcal{H}'$. For $1\le q\le n$, we have the natural decomposition
\begin{equation}\label{eq:nonzero-degree-block-decomposition}
    \mathcal{M}_q'=\mathcal{M}_q\oplus\mathbb{R}E_{n-q,n}.
\end{equation}
Here we set $\mathcal{M}_n=0$.

\begin{lemma}\label{lem:nonzero-degree-block-decomposition}
    For $1\le q\le n$, the operator $L_{\mathbf{A}',q}$ preserves the decomposition \eqref{eq:nonzero-degree-block-decomposition} and under the natural identification $\mathbb{R}E_{n-q,n}\cong\mathbb{R}$, we have
    \[
     L_{\mathbf{A}',q} = L_{\mathbf{A},q}\oplus\left(v\mapsto g_{n-q}(\mathbf{A})v\right).
    \]
\end{lemma}

\begin{proof}
    The first assertion follows from the property of the multiplication of matrix units and \eqref{eq:nonzero-diagonal-embed}. For the second assertion, note that $E_{n-q,n}A_i'=0$ for $-(n-2)\le i\le n-2$, and therefore
    \begin{align*}
         L_{\mathbf{A}'}(E_{n-q,n})&=\sum_{\sigma\in S_{2n-3}} \sgn(\sigma)A_{\sigma(-(n-2))}'\cdots A_{\sigma(n-2)}'E_{n-q,n}\\
         &=s_{2n-3}(\mathbf{A}')E_{n-q,n}\\
         &=\operatorname{diag}\{g_0(\mathbf{A}),\cdots,g_{n-1}(\mathbf{A}),0\}E_{n-q,n}\\
         &=g_{n-q}(\mathbf{A})E_{n-q,n}.
    \end{align*}
\end{proof}

\begin{proposition}\label{prop:additional-nonzero-nonempty}
For every $-n\le q\le n,~q\ne0$, there exists a tuple $\mathbf{A}'\in\mathcal{H}'$ which makes equality hold in \eqref{eq:block-nullity-lower-bound-additional}.
\end{proposition}

\begin{proof}
By \eqref{eq:additional-opposite-block-ranks}, we can assume $1\le q\le n$. By Lemma \ref{lem:block-nullity-lower-bound} and Proposition \ref{prop:nonzero-block}, those tuples $\mathbf{A}\in\mathcal{H}$ which make equality hold in \eqref{eq:block-nullity-lower-bound} form a nonempty Zariski-open subset of $\mathcal{H}$. By Lemma \ref{lem:g_a-nonzero}, the condition $g_{n-q}(\mathbf{A})\ne0$ also defines a nonempty Zariski-open subset of $\mathcal{H}$. Since $\mathcal{H}$ is irreducible, these two nonempty open subsets intersect. For a tuple $\mathbf{A}$ in the intersection, let $\mathbf{A}'\in\mathcal{H}'$ be the associated tuple as in \eqref{eq:nonzero-diagonal-embed}. Then by Lemma \ref{lem:nonzero-degree-block-decomposition}, for this $\mathbf{A}'$, the equality in \eqref{eq:block-nullity-lower-bound-additional} holds.
\end{proof}

\subsection{The zero-degree block}\label{subsec:additional-zero}

Throughout this subsection, we assume $n\ge4$. Set $s=\left\lfloor\frac{n-2}{2}\right\rfloor>0$ and a reflection $r(l)=n-1-l$ as in Subsection \ref{subsec:zero-block}. Let $\mathbf{A}(t)$ be the family of tuples in $\mathcal{H}$ parametrized by $t\in\mathbb{R}$ specialized as in \eqref{eq:zero-block-family} and \eqref{eq:zero-block-middle-pair}. For $0\le i\le n-1$, put $Z_i=E_{i,i}$ and let
$W=\Span\{Z_1,\cdots,Z_s,Z_{r(1)},\cdots,Z_{r(s)}\}\subseteq\mathcal{M}_0$. Let $M(t)$ be the matrix of $w_{\mathbf{A}(t),0}|_W$ under the ordered basis $Z_1,\cdots,Z_s,Z_{r(1)},\cdots,Z_{r(s)}$. By Proposition \ref{prop:zero-block-first-order} and the discussion below it, $\det M(t)$ is a nonzero polynomial in $\mathbb{R}[t]$.

If $n$ is even, choose $t_0\in\mathbb{R}$ such that $\det M(t_0)\ne0$. Specialize $\mathbf{A}=\mathbf{A}(t_0)$ and let $\mathbf{A}'\in\mathcal{H}'$ be the associated tuple as in \eqref{eq:nonzero-diagonal-embed}. Then
\[
\rank w_{\mathbf{A}',0}|_W=\rank w_{\mathbf{A},0}|_W=n-2
\]
and hence
\[
\dim\ker L_{\mathbf{A}',0}\le n+1-(n-2)=3.
\]
Therefore the equality in \eqref{eq:block-nullity-lower-bound-additional} holds.

In the remaining part of this subsection, we assume $n\ge5$ is odd. Then $s=\frac{n-3}{2}$. We also set $h=s+1=\frac{n-1}{2}$.

\begin{lemma}\label{lem:Z_h-radical}
For every $t\in\mathbb{R}$, we have $Z_h\in\ker L_{\mathbf{A}(t),0}$.
\end{lemma}

\begin{proof}
It suffices to show $w_{\mathbf{A}(t),0}(Z_j,Z_h)=0$ for every $0\le j\le n-1$. We can assume $j\ne h$ because $w_{\mathbf{A}(t),0}$ is alternating. Expand the two-term matrices in \eqref{eq:zero-block-family} into matrix units as in the proof of Proposition \ref{prop:zero-block-first-order}. By Lemma \ref{lem:trace-euler}, we convert the problem to the sum of the signs of Eulerian circuits starting from $0$ of some directed graphs. Each graph contains two circuits $0\to0$ and $0\to h\to h\to0$, whose numbers of edges are odd. Moreover, any Eulerian circuit starting from $0$ must contain the two circuits. If we exchange the two circuits, we will get a new Eulerian circuit with opposite sign, which makes the entire sum zero.
\end{proof}

We now pass to dimension $n+1$. Let $\mathbf{A}'(t)\in\mathcal{H}'$ be as in \eqref{eq:nonzero-diagonal-embed}. We introduce a new parameter $\lambda\in\mathbb{R}$. Define
\[
\mathbf{A}'(t,\lambda)\in \mathcal{H}'
\]
by $A_i'(t,\lambda)=A_i'(t)$ for all $i$ except for the following two positions
\[
A_{-h}'(t,\lambda)=A_{-h}'+\lambda E_{n,h+1},\quad A_{h+1}'(t,\lambda)=A_{h+1}'(t)+\lambda E_{h,n}.
\]
In particular, we have $\mathbf{A}'(t,0)=\mathbf{A}'(t)$. Put $Z_n=E_{n,n}\in\Mat_{n+1}(\mathbb{R}).$

\begin{lemma}\label{lem:Z_n-radical}
For every $t\in\mathbb{R}$, we have $Z_n\in\ker L_{\mathbf{A}'(t),0}$.
\end{lemma}

\begin{proof}
    This follows from block embedding and matrix multiplication.
\end{proof}

\begin{lemma}\label{lem:lambda-coefficient}
For every $X,Y\in\mathcal{M}_0'$, the coefficient of $\lambda$ in $w_{\mathbf{A}'(t,\lambda),0}(X,Y)$ is zero.
\end{lemma}

\begin{proof}
Every term linear in $\lambda$ contains exactly one of the two perturbation edges $E_{n,h+1}$ and $E_{h,n}$. Consequently, for the vertex $n$ of each corresponding directed graph, the outdegree is not equal to the indegree, and hence the graph admits no Eulerian circuit.
\end{proof}

\begin{lemma}\label{lem:t-square-coefficient}
The coefficient of $\lambda^2$ of the polynomial $w_{\mathbf{A}'(0,\lambda),0}(Z_h,Z_n)$ is nonzero.
\end{lemma}

\begin{proof}
Every term contributing to the coefficient of $\lambda^2$ must contain both perturbation edges $E_{n,h+1}$ and $E_{h,n}$. For $1\le l\le s$ and $A_{-l}'(t,\lambda)=A_{-l}'=E_{l,0}+E_{n-1,r(l)}$, by analyzing the indegree and outdegree of vertices $l$ similarly to the proof of Proposition \ref{prop:zero-block-first-order}, we obtain that only the edges $E_{l,0}$ can contribute to a directed graph admitting an Eulerian circuit. Therefore, the edges of the corresponding directed graph are $0\to0$, $0\to h\to h\to n\to n\to h+1\to0$ and $0\to j,~j\to0$ for $1\le j\le n-2,~j\ne h,h+1$. The vertex $n-1$ is isolated. Result follows from Lemma \ref{lem:odd-circuit} with a little variation.
\end{proof}

\begin{proposition}
For $n\ge5$ odd, there exist $t_0,\lambda_0\in\mathbb{R}$, such that the tuple $\mathbf{A}'(t_0,\lambda_0)\in\mathcal{H}'$ makes equality hold in \eqref{eq:block-nullity-lower-bound-additional}.
\end{proposition}

\begin{proof}
Let $W'=W\oplus
\operatorname{Span}\{Z_h,Z_n\}$. Consider $w_{\mathbf{A}'(t,\lambda),0}|_{W'}$. By Lemma \ref{lem:Z_h-radical},~\ref{lem:Z_n-radical} and \ref{lem:lambda-coefficient}, under the standard basis of $W$ followed by $Z_h,Z_n$, the matrix of the bilinear form $w_{\mathbf{A}'(t,\lambda),0}|_{W'}$ has the shape
\[
M'(t,\lambda)=\begin{pmatrix}
M(t)+\lambda^2*
&
\lambda^2*
&
\lambda^2*
\\
\lambda^2*
&
0
&
\lambda^2f(t,\lambda)
\\
\lambda^2*
&
-\lambda^2f(t,\lambda)
&
0
\end{pmatrix},
\]
where $\lambda^2f(t,\lambda)=w_{\mathbf{A}'(t,\lambda),0}(Z_h,Z_n)$ and by Lemma \ref{lem:t-square-coefficient}, we have $f(0,0)\ne0$. Choose $t_0\in\mathbb{R}$ such that $\det M(t_0),f(t_0,0)\ne0$. Then
\[
\det M'(t_0,\lambda)=\det M(t_0)\cdot f(t_0,\lambda)^2\lambda^4(\mod \lambda^5).
\]
In particular, $\det M'(t_0,\lambda)$ is a nonzero polynomial in $\mathbb{R}[\lambda]$, so there exists $\lambda_0\in \mathbb{R}$ such that
\[
\det M'(t_0,\lambda_0)\ne0.
\]
For this specialization, we have
\[
\rank L_{\mathbf{A}'(t_0,\lambda_0),0}=\rank w_{\mathbf{A}'(t_0,\lambda_0),0}\ge\rank w_{\mathbf{A}'(t_0,\lambda_0),0}|_{W'}=n-1.
\]
Equivalently, we have $\dim\ker L_{\mathbf{A}'(t_0,\lambda_0),0}\le2$. Therefore the equality in \eqref{eq:block-nullity-lower-bound-additional} holds.
\end{proof}

\begin{remark}
The assumption $n\ge5$ in the preceding construction is necessary. When $n=3$, so that $(k,n+1)=(3,4)$, for every tuple $\mathbf{A}'=(A_{-1}',A_0',A_1')\in\mathcal{H}'$, we have $L_{\mathbf{A}',0}=0$. Indeed, for $X\in \mathcal{M}_0'$, we have 
\begin{align*}
L_{\mathbf{A}',0}(X)&=s_4(A_{-1}',A_0',A_1',X)\\
&=\left(\sum_{\sigma\in S_3}\sgn(\sigma)A_{\sigma(-1)}'XA_{\sigma(0)}'A_{\sigma(1)}'-\sum_{\sigma\in S_3}\sgn(\sigma)A_{\sigma(-1)}'A_{\sigma(0)}'XA_{\sigma(1)}'\right)\\
&+\left(\sum_{\sigma\in S_3}\sgn(\sigma)A_{\sigma(-1)}'A_{\sigma(0)}'A_{\sigma(1)}'X-\sum_{\sigma\in S_3}\sgn(\sigma)XA_{\sigma(-1)}'A_{\sigma(0)}'A_{\sigma(1)}'\right)\\
&=\left(-\sum_{\sigma\in S_2}\sgn(\sigma)A_0'XA_{\sigma(-1)}'A_{\sigma(1)}'+\sum_{\sigma\in S_2}\sgn(\sigma)A_{\sigma(-1)}'A_{\sigma(1)}'XA_0'\right)\\
&+\left(\sum_{\sigma\in S_2}\sgn(\sigma)A_{\sigma(-1)}'XA_0'A_{\sigma(1)}'-\sum_{\sigma\in S_2}\sgn(\sigma)A_{\sigma(-1)}'A_0'XA_{\sigma(1)}'\right)\\
&+\left(-\sum_{\sigma\in S_2}\sgn(\sigma)A_{\sigma(-1)}'XA_{\sigma(1)}'A_0'+\sum_{\sigma\in S_2}\sgn(\sigma)A_0'A_{\sigma(-1)}'XA_{\sigma(1)}'\right)\\
&=0.
\end{align*}
Here we use that any two matrices in $\mathcal{M}_0'$ commute.
\end{remark}

\subsection{Proof of Theorem \ref{thm:additional-case}}

\begin{proof}[Proof of Theorem \ref{thm:additional-case}]
For $n\ge4$, the result follows from \eqref{eq:block-nullity-lower-bound-additional}, Proposition \ref{prop:additional-nonzero-nonempty} and Subsection \ref{subsec:additional-zero}; see the argument in Subsection \ref{subsec:proof-of-thm-boudary}. 

It remains to consider $n=3$, i.e. $(k,n+1)=(3,4)$. We give an explicit good tuple $\mathbf{A}=(A_1,A_2,A_3)\in\Mat_4(\mathbb{R})^3$. 
Let $A_1=E_{1,1},A_2=E_{0,2}+E_{2,1}+E_{3,0}$ and $A_3=E_{0,1}+E_{1,3}$. We claim that $\ker L_\mathbf{A}=\{I_4,A_1,A_2,A_3\}$. 

Decompose $\Mat_4(\mathbb{R})$ as $\Mat_4(\mathbb{R})=V_0\oplus V_1\oplus V_2 \oplus V_3\oplus V_4$, where
\[
\begin{aligned}
    V_0&=\Span\{E_{0,0},E_{1,1},E_{2,2},E_{3,3}\},\\
    V_1&=\Span\{E_{0,1},E_{1,3},E_{3,2}\},\\
    V_2&=\Span\{E_{0,3},E_{1,2},E_{2,0}\},\\
    V_3&=\Span\{E_{0,2},E_{2,1},E_{3,0}\},\\
    V_4&=\Span\{E_{1,0},E_{2,3},E_{3,1}\}.
\end{aligned}
\]

By definition of $s_4$ and calculation, we find 
$$L_\mathbf{A}(V_0)\subseteq V_4,\quad L_\mathbf{A}(V_1)\subseteq V_0,\quad L_\mathbf{A}(V_2)\subseteq V_1,\quad L_\mathbf{A}(V_3)\subseteq V_2,\quad L_\mathbf{A}(V_4)\subseteq V_3.$$
Hence, we have $\ker L_\mathbf{A}=\bigoplus_{i=0}^4\ker( L_\mathbf{A}|_{V_i})$. 

We now determine these five kernels. Let $X=aE_{0,0}+bE_{1,1}+cE_{2,2}+dE_{3,3}\in V_0$. Then $L_\mathbf{A}(X)=(-a+d)E_{1,0}+(c-d)E_{2,3}+(a-d)E_{3,1}.$
Thus, $L_\mathbf{A}(X)=0$ if and only if $a=c=d$ and $b$ is arbitrary. Therefore, 
$$\ker(L_\mathbf{A}|_{V_0})=\Span\{I_4,E_{1,1}\}=\Span\{I_4,A_1\}.$$

Using the same method to calculate the kernels, we have $$\ker(L_\mathbf{A}|_{V_1})=\Span\{A_3\},\quad\ker(L_\mathbf{A}|_{V_2})=0,\quad\ker(L_\mathbf{A}|_{V_3})=\Span\{A_2\},\quad\ker(L_\mathbf{A}|_{V_4})=0.$$
Therefore, \[ \ker L_\mathbf{A} = \Span\{I_4,A_1,A_2,A_3\}. \]
Since $I_4,A_1,A_2,A_3$ are linearly independent, we have $\dim\ker L_\mathbf{A}=4$.  Thus $\mathbf{A}$ is a good tuple for $(k,n+1)=(3,4)$.
\end{proof}

\section{Dimension extensions}
\label{sec:extensions}

Throughout this section, assume that $n\geq 3$, $k$ is odd and $3\le k\le 2n-3$. The goal of this section is to prove that if a good tuple exists for $(k,n)$, then a good tuple also exists for $(k,n+2)$.

\begin{lemma}\label{lem:det-sm-nonzero}
Let $2\le m\le 2n-2$ be an integer. There exists $\mathbf{A}\in\Mat_n(\mathbb{R})^m$ such that $\det s_m(\mathbf{A})\ne 0$. 
\end{lemma}

\begin{proof}
The proof is similar to that of Lemma \ref{lem:g_a-nonzero}. Denote $m'=\left\lfloor\frac{m}{2}\right\rfloor$. We first assume that $m$ is odd. Specialize 
\[
\mathbf{A}=(A_{-m'},\cdots,A_{m'})\in\mathcal{H}_m=\prod_{i=-m'}^{m'}\mathcal{M}_i.
\]
Then
\[
s_m(\mathbf{A})=\operatorname{diag}
\{
g_0'(\mathbf{A}),\cdots,g_{n-1}'(\mathbf{A})
\}
\]
for polynomial functions $g_i'$ on $\mathcal{H}_m$. It suffices to show each $g_i'$ is nonzero. Set $l=\max\{0,i-m'\}$. Specialize $A_0=E_{l,l}$ and $A_j=E_{l,l+j},~A_{-j}=E_{l+j,l}$ for $1\le j\le m'$. Then by Lemma \ref{lem:trace-euler} and \ref{lem:bidirected-tree}, we have $g_i'(\mathbf{A})\ne 0$.

Assume that $m$ is even. Specialize $\mathbf{A}=(A_{-m'},\cdots,A_{-1},A_1,\cdots,A_{m'})$, where $A_i\in\mathcal{M}_i$. Then
\[
s_m(\mathbf{A})=\operatorname{diag}
\{
f_0'(\mathbf{A}),\cdots,f_{n-1}'(\mathbf{A})
\}
\]
for polynomial functions $f_i'$ on $\prod_{j=-m'}^{-1}\mathcal{M}_j\times\prod_{j=1}^{m'}\mathcal{M}_j$. It suffices to show each $f_i'$ is nonzero. Set $l=\max\{0,i-m'\}$. Specialize $A_j=E_{l,l+j},~A_{-j}=E_{l+j,l}$ for $1\le j\le m'$. Then by \cite[Lemma 2.3]{BrassilReichstein}, we have $f_i'(\mathbf{A})\ne 0$.
\end{proof}

For a tuple $\mathbf{A}=(A_1,\cdots,A_k)\in\Mat_n(\mathbb{R})^k$, write $\bar{\mathbf{A}}=(A_1,\cdots,A_{k-1})\in\Mat_n(\mathbb{R})^{k-1}$.

\begin{lemma}\label{lem:specialization}
    Let $k$ be odd and $3\le k\le 2n-3$. Assume that a good tuple exists for $(k,n)$. Then there exists a good tuple $\mathbf{A}=(A_1,\cdots,A_k)\in\Mat_n(\mathbb{R})^k$ such that 
    \begin{enumerate}
    \item[(a)] $\det s_k(\mathbf{A})\ne 0$;
    \item[(b)] $\det s_{k-1}(\bar{\mathbf{A}})\ne 0$;
    \item[(c)] $\dim\Span\{s_k(\mathbf{A}),s_{k-1}(\bar{\mathbf{A}})\}=2$.
\end{enumerate}
\end{lemma}

\begin{proof}
    By Lemma \ref{lem:det-sm-nonzero} and Corollary \ref{nullity-lower-bound}, good tuples satisfying conditions (a) and (b) form a nonempty Zariski-open subset of $\Mat_n(\mathbb{R})^k$. As tuples satisfying condition (c) also form a Zariski-open subset of $\Mat_n(\mathbb{R})^k$, it suffices to show the subset is nonempty. 
    
    For generic $A_1,\cdots,A_{k-1}$, by \cite[Theorem 1.2]{BrassilReichstein}, the nullity of $L_{\bar{\mathbf{A}}}$ is $k-1$. Thus $\dim\Im L_{\bar{\mathbf{A}}}=n^2-(k-1)>1$. Therefore, we can choose $A_k\in\Mat_n(\mathbb{R})$ such that 
    \[
    0\ne s_k(\mathbf{A})=L_{\bar{\mathbf{A}}}(A_k)\notin \mathbb{R}L_{\bar{\mathbf{A}}}(I_n)=\mathbb{R}s_{k-1}(\bar{\mathbf{A}}).
    \]
    Here the last equality follows from Lemma \ref{lem:identity-insertion}. We can further require $s_{k-1 (\bar{\mathbf{A}})}\ne 0$. Such a tuple $\mathbf{A}$ satisfies $\dim\Span\{s_k(\mathbf{A}),s_{k-1}(\bar{\mathbf{A}})\}=2$.
\end{proof}

For a tuple $\mathbf{A}=(A_1,\cdots,A_k)\in\Mat_n(\mathbb{R})^k$, set 
\begin{equation}\label{eq:diagonal-embed-n+2}
    \tilde{A}_i=\operatorname{diag}\{A_i,0,0\}~(1\le i\le k-1),\quad \tilde{A}_k=\operatorname{diag}\{A_k,0,\tau\}
\end{equation}
for some $\tau\in\mathbb{R}$ and write $\tilde{\mathbf{A}}=(\tilde{A}_1,\cdots,\tilde{A}_k)\in\Mat_{n+2}(\mathbb{R})^k$. We introduce the four subspaces of $\Mat_{n+2}(\mathbb{R})$:
\[
\mathcal{X}=\left\{\begin{pmatrix}X&0\\0&0_{2\times 2}\end{pmatrix}~\middle|~X\in\Mat_n(\mathbb{R})\right\},
\]
\[
\mathcal{S}=\left\{\begin{pmatrix}0_{n\times n}&0\\0&Z\end{pmatrix}~\middle|~Z\in\Mat_2(\mathbb{R})\right\},
\]
\[
\mathcal{U}=\left\{\begin{pmatrix}0_{n\times n}&(u_1,u_2)\\0&0_{2\times 2}\end{pmatrix}~\middle|~u_1,u_2\in \mathbb{R}^n \right\},
\]
\[
\mathcal{V}=\left\{\begin{pmatrix}0_{n\times n}&0\\(v_1,v_2)^\mathsf{T}&0_{2\times 2} \end{pmatrix}~\middle|~v_1,v_2\in\mathbb{R}^n\right\}. 
\]
Then
\begin{equation}\label{eq:four-space-decomposition-n+2}
    \Mat_{n+2}(\mathbb{R}) = \mathcal{X}\oplus\mathcal{S}\oplus\mathcal{U}\oplus\mathcal{V}.
\end{equation}
  
\begin{lemma}\label{lem:block-decomposition-n+2} 
Each of the four spaces $\mathcal{X},~\mathcal{S},~\mathcal{U}$ and $\mathcal{V}$ is invariant under $L_{\tilde{\mathbf{A}}}$. Under the natural identifications 
\[ 
\mathcal{X}\cong\Mat_n(\mathbb{R}),\quad\mathcal{S}\cong\Mat_2(\mathbb{R}),\quad\mathcal{U}\cong\mathbb{R}^n\oplus\mathbb{R}^n,\quad\mathcal{V}\cong\mathbb{R}^n\oplus\mathbb{R}^n, 
\]
we have $L_{\tilde{\mathbf{A}}}|_{\mathcal{X}} = L_{\mathbf{A}},~L_{\tilde{\mathbf{A}}}|_{\mathcal{S}}=0$ and
\begin{align*}
L_{\tilde{\mathbf{A}}}|_{\mathcal{U}}(u_1,u_2)&=\left(s_{k}(\mathbf{A})u_1,~\left(s_{k}(\mathbf{A})-\tau s_{k-1}(\bar{\mathbf{A}})\right)u_2\right),\\
L_{\tilde{\mathbf{A}}}|_{\mathcal{V}}(v_1,v_2)&=\left(-s_{k}(\mathbf{A})^{\mathsf{T}}v_1,~\left(-s_{k}(\mathbf{A})^{\mathsf{T}}+\tau s_{k-1}(\bar{\mathbf{A}})^{\mathsf{T}}\right)v_2\right). 
\end{align*}
\end{lemma} 

\begin{proof}
Every $\tilde{A}_i$ is block-diagonal. Hence left or right multiplication by $\tilde{A}_i$ preserves each summand in \eqref{eq:four-space-decomposition-n+2}, so does $L_{\tilde{\mathbf{A}}}$. For $X\in\Mat_n(\mathbb{R})$ and $Y\in\Mat_2(\mathbb{R})$, block multiplication gives 
\[
L_{\tilde{\mathbf{A}}} \left(\begin{pmatrix} X&0\\ 0&0_{2\times2}\end{pmatrix}\right) = \begin{pmatrix} L_{\mathbf{A}}(X)&0\\ 0&0_{2\times2} \end{pmatrix},\quad L_{\tilde{\mathbf{A}}} \left(\begin{pmatrix} 0_{n\times n}&0\\ 0&Y \end{pmatrix}\right) = 0. 
\]
For
\[ 
U(u_1,u_2)= 
\begin{pmatrix} 
0_{n\times n}&(u_1,u_2)\\
0&0_{2\times 2} 
\end{pmatrix} 
\in\mathcal{U}
\] 
and $1\le i\le k-1$, we have $U(u_1,u_2)\tilde{A}_i=0$. Therefore
\begin{align*}
    L_{\tilde{\mathbf{A}}}\left(U(u_1,u_2)\right)&=\sum_{\sigma\in S_k} \sgn(\sigma)\tilde{A}_{\sigma(1)}\cdots \tilde{A}_{\sigma(k)}U(u_1,u_2)-\sum_{\sigma\in S_{k-1}} \sgn(\sigma)\tilde{A}_{\sigma(1)}\cdots \tilde{A}_{\sigma(k-1)}U(u_1,u_2)\tilde{A}_k\\
    &= 
    \begin{pmatrix} 0_{n\times n}&s_k(\mathbf{A})(u_1,u_2)\\0&0_{2\times 2}
    \end{pmatrix}
    -
    \begin{pmatrix} 0_{n\times n}&\left(0,\tau s_{k-1}(\bar{\mathbf{A}})(u_2)\right)\\0&0_{2\times 2}
    \end{pmatrix}
    \\
    &=U\left(s_{k}(\mathbf{A})u_1,~\left(s_{k}(\mathbf{A})-\tau s_{k-1}(\bar{\mathbf{A}})\right)u_2\right). 
\end{align*} 
The computation of $L_{\tilde{\mathbf{A}}}|_{\mathcal{V}}$ is similar. 
\end{proof}

To control the nullity, we introduce the perturbation term. Let 
\begin{equation}\label{eq:perturbation}
    B=
\begin{pmatrix} 
0_{n\times n}&U\\V&0_{2\times 2}
\end{pmatrix}
\in\Mat_{n+2}(\mathbb{R})
\end{equation}
for some $U\in\Mat_{n\times2}(\mathbb{R}),~V\in\Mat_{2\times n}(\mathbb{R})$ and write $\mathbf{B}=(\tilde{A}_1,\cdots,{\tilde{A}_{k-1}},B)\in\Mat_{n+2}(\mathbb{R})^k$. The following lemma follows immediately from matrix multiplication.

\begin{lemma}\label{lem:perturbation-term-V}
    We have 
    \[
    L_{\mathbf{B}}(\mathcal{X})\subseteq\mathcal{U}\oplus\mathcal{V},\quad L_{\mathbf{B}}(\mathcal{U})\subseteq\mathcal{X}\oplus\mathcal{S},\quad L_{\mathbf{B}}(\mathcal{V})\subseteq\mathcal{S}\oplus\mathcal{X},\quad L_{\mathbf{B}}(\mathcal{S})\subseteq\mathcal{U}\oplus\mathcal{V}.
    \]
    Moreover, for $Z\in\Mat_2(\mathbb{R})$, we have
    \begin{equation}\label{eq:L-B-Z}
        L_{\mathbf{B}}\left(\begin{pmatrix}0_{n\times n}&0\\0&Z\end{pmatrix}\right)=\begin{pmatrix}0_{n\times n}&s_{k-1}(\bar{\mathbf{A}})UZ\\-ZVs_{k-1}(\bar{\mathbf{A}})&0_{2\times 2}\end{pmatrix}.
    \end{equation}
\end{lemma}

\begin{proposition}\label{prop:n-to-n+2}
    Let $k$ be odd and $3\le k\le 2n-3$. Assume that a good tuple exists for $(k,n)$. Then a good tuple also exists for $(k,n+2)$.
\end{proposition}

\begin{proof}
We choose a good tuple $\mathbf{A}=(A_1,\cdots,A_k)\in\Mat_n(\mathbb{R})^k$ satisfying conditions (a), (b), (c) in Lemma \ref{lem:specialization} and $0\ne\tau\in\mathbb{R}$ such that $\det\left(s_k(\mathbf{A})-\tau s_{k-1}(\bar{\mathbf{A}})\right)\ne0$. Associate the tuple $\mathbf{A}$ with the tuple $\tilde{\mathbf{A}}=(\tilde{A}_1,\cdots,\tilde{A}_k)\in\Mat_{n+2}(\mathbb{R})^k$ as in \eqref{eq:diagonal-embed-n+2}. By Lemma \ref{lem:block-decomposition-n+2}, both $L_{\tilde{\mathbf{A}}}|_{\mathcal{U}}$ and $L_{\tilde{\mathbf{A}}}|_{\mathcal{V}}$ are linear isomorphisms. Let 
\[
\tilde{\mathbf{A}}(t)=(\tilde{A}_1,\cdots,{\tilde{A}_{k-1}},\tilde{A}_k+tB),
\]
where $t\in\mathbb{R}$ and $B$ is as in \eqref{eq:perturbation}. Then $L_{\tilde{\mathbf{A}}(t)}=L_{\tilde{\mathbf{A}}}+tL_{\mathbf{B}}$. Our goal is to find $t\in\mathbb{R},~U\in\Mat_{n\times2}(\mathbb{R})$ and $V\in\Mat_{2\times n}(\mathbb{R})$ such that $\tilde{\mathbf{A}}(t)$ is a good tuple.

By \eqref{eq:rad-kernel}, we can turn to $w_{\tilde{\mathbf{A}}(t)}=w_{\tilde{\mathbf{A}}}+tw_{\mathbf{B}}$. Choose a complementary space $\mathcal{F}$ of $\rad w_{\mathbf{A}}$ in $\Mat_n(\mathbb{R})$. Then $w_{\mathbf{A}}|_{\mathcal{F}}$ is nondegenerate and $\rank w_{\mathbf{A}}=\dim \mathcal{F}$. Write $w_{\tilde{\mathbf{A}}(t)}|_{\mathcal{F}\oplus\mathcal{U}\oplus\mathcal{V}\oplus\mathcal{S}}$ in blocks with respect to the decomposition $\mathcal{F}\oplus\mathcal{U}\oplus\mathcal{V}\oplus\mathcal{S}$. By Lemma \ref{lem:block-decomposition-n+2}, \ref{lem:perturbation-term-V}, the blocks are like
\begin{equation}\label{eq:block}
    \begin{pmatrix}
w_{\mathbf{A}}|_{\mathcal{F}}&tw_{\mathbf{B}}|_{\mathcal{F}\times\mathcal{U}}&tw_{\mathbf{B}}|_{\mathcal{F}\times\mathcal{V}}&0\\tw_{\mathbf{B}}|_{\mathcal{U}\times\mathcal{F}}&0&w_{\tilde{\mathbf{A}}}|_{\mathcal{U}\times\mathcal{V}}&tw_{\mathbf{B}}|_{\mathcal{U}\times\mathcal{S}}\\tw_{\mathbf{B}}|_{\mathcal{V}\times\mathcal{F}}&w_{\tilde{\mathbf{A}}}|_{\mathcal{V}\times\mathcal{U}}&0&tw_{\mathbf{B}}|_{\mathcal{V}\times\mathcal{S}}\\0&tw_{\mathbf{B}}|_{\mathcal{S}\times\mathcal{U}}&tw_{\mathbf{B}}|_{\mathcal{S}\times\mathcal{V}}&0
\end{pmatrix}.
\end{equation}
Here, $w_{\tilde{\mathbf{A}}}|_{\mathcal{U}\times\mathcal{V}}$ (resp. $w_{\tilde{\mathbf{A}}}|_{\mathcal{V}\times\mathcal{U}}$) is nondegenerate because $L_{\tilde{\mathbf{A}}}|_{\mathcal{V}}$ (resp. $L_{\tilde{\mathbf{A}}}|_{\mathcal{U}}$) is a linear isomorphism and the trace pairing on $\mathcal{U}\times\mathcal{V}$ is perfect. To compute the rank of \eqref{eq:block}, it is convenient to identify $w_{\tilde{\mathbf{A}}(t)}|_{\mathcal{F}\oplus\mathcal{U}\oplus\mathcal{V}\oplus\mathcal{S}}$ with
\[
\mathcal{F}\oplus\mathcal{U}\oplus\mathcal{V}\oplus\mathcal{S}\longrightarrow\mathcal{F}^*\oplus\mathcal{U}^*\oplus\mathcal{V}^*\oplus\mathcal{S}^*.
\]
We perform row operations on \eqref{eq:block}, using the isomorphisms
\[
w_{\tilde{\mathbf{A}}}|_{\mathcal{U}\times\mathcal{V}}:\mathcal{U}\rightarrow\mathcal{V}^*,\quad w_{\tilde{\mathbf{A}}}|_{\mathcal{V}\times\mathcal{U}}:\mathcal{V}\rightarrow\mathcal{U}^*
\]
to cancel out $tw_{\mathbf{B}}|_{\mathcal{S}\times\mathcal{V}}$ and $tw_{\mathbf{B}}|_{\mathcal{S}\times\mathcal{U}}$, respectively. We obtain
\begin{equation}\label{eq:row-operation}
     \begin{pmatrix}
w_{\mathbf{A}}|_{\mathcal{F}}&tw_{\mathbf{B}}|_{\mathcal{F}\times\mathcal{U}}&tw_{\mathbf{B}}|_{\mathcal{F}\times\mathcal{V}}&0\\tw_{\mathbf{B}}|_{\mathcal{U}\times\mathcal{F}}&0&w_{\tilde{\mathbf{A}}}|_{\mathcal{U}\times\mathcal{V}}&tw_{\mathbf{B}}|_{\mathcal{U}\times\mathcal{S}}\\tw_{\mathbf{B}}|_{\mathcal{V}\times\mathcal{F}}&w_{\tilde{\mathbf{A}}}|_{\mathcal{V}\times\mathcal{U}}&0&tw_{\mathbf{B}}|_{\mathcal{V}\times\mathcal{S}}\\t^2*&0&0&-t^2 w
\end{pmatrix},
\end{equation}
where 
\begin{equation}\label{eq:w-expression}
    w=w_{\mathbf{B}}|_{\mathcal{U}\times\mathcal{S}}(w_{\tilde{\mathbf{A}}}|_{\mathcal{U}\times\mathcal{V}})^{-1}w_{\mathbf{B}}|_{\mathcal{S}\times\mathcal{V}}+w_{\mathbf{B}}|_{\mathcal{V}\times\mathcal{S}}(w_{\tilde{\mathbf{A}}}|_{\mathcal{V}\times\mathcal{U}})^{-1}w_{\mathbf{B}}|_{\mathcal{S}\times\mathcal{U}}
\end{equation}
is a bilinear form on $\mathcal{S}$. Now assume that we can choose $U\in\Mat_{n\times2}(\mathbb{R})$ and $V\in\Mat_{2\times n}(\mathbb{R})$ such that $w$ is nondegenerate. After choosing bases of $\mathcal{F},~\mathcal{U},~\mathcal{V}$ and $\mathcal{S}$ respectively, the determinant of the matrix corresponding to \eqref{eq:row-operation} is congruent to
\[
t^8\det   \begin{pmatrix}
w_{\mathbf{A}}|_{\mathcal{F}}&0&0&0\\0&0&w_{\tilde{\mathbf{A}}}|_{\mathcal{U}\times\mathcal{V}}&w_{\mathbf{B}}|_{\mathcal{U}\times\mathcal{S}}\\0&w_{\tilde{\mathbf{A}}}|_{\mathcal{V}\times\mathcal{U}}&0&w_{\mathbf{B}}|_{\mathcal{V}\times\mathcal{S}}\\0&0&0&w
\end{pmatrix}
(\mod t^9),
\]
which is nonzero for $t\ne 0$ because $w_{\mathbf{A}}|_{\mathcal{F}},~w_{\tilde{\mathbf{A}}}|_{\mathcal{U}\times\mathcal{V}},~w_{\tilde{\mathbf{A}}}|_{\mathcal{V}\times\mathcal{U}}$ and $w$ are all nondegenerate. This implies that there exists $0\ne t\in\mathbb{R}$ such that the determinant of the matrix corresponding to \eqref{eq:row-operation} is nonzero. Therefore $w_{\tilde{\mathbf{A}}(t)}|_{\mathcal{F}\oplus\mathcal{U}\oplus\mathcal{V}\oplus\mathcal{S}}$ is nondegenerate for such $t$ and
\[
\dim\ker w_{\tilde{\mathbf{A}}(t)}\le n^2-\dim \mathcal{F}=\dim\ker w_{\mathbf{A}}.
\]
By Corollary \ref{nullity-lower-bound}, $\tilde{\mathbf{A}}(t)$ is a good tuple for $(k,n+2)$.

It remains to show that we can choose $U\in\Mat_{n\times2}(\mathbb{R})$ and $V\in\Mat_{2\times n}(\mathbb{R})$ such that $w$ is nondegenerate. Now we identify $\mathcal{U},~\mathcal{V},~\mathcal{S}$ with $\Mat_{n\times2}(\mathbb{R}),~\Mat_{2\times n}(\mathbb{R}),~\Mat_{2\times2}(\mathbb{R})$ respectively. For $Z\in\mathcal{S}$, by definition $(w_{\tilde{\mathbf{A}}}|_{\mathcal{U}\times\mathcal{V}})^{-1}w_{\mathbf{B}}|_{\mathcal{S}\times\mathcal{V}}(Z)$ is the unique element $U'=(u_1',u_2')\in\mathcal{U}$ satisfying
\[
w_{\tilde{\mathbf{A}}}(U',V')=w_{\mathbf{B}}(Z,V'),\quad\forall ~V'\in\mathcal{V}.
\]
We have
\[
w_{\tilde{\mathbf{A}}}(U',V')=-w_{\tilde{\mathbf{A}}}(V',U')=-\tr \left(V'L_{\tilde{\mathbf{A}}}(U')\right)
\]
and by Lemma \ref{lem:perturbation-term-V}
\[
w_{\mathbf{B}}(Z,V')=-w_{\mathbf{B}}(V',Z)=-\tr\left(V'L_{\mathbf{B}}(Z)\right)=-\tr\left(V's_{k-1}(\bar{\mathbf{A}})UZ\right).
\]
Therefore, we have $L_{\tilde{\mathbf{A}}}(U')=s_{k-1}(\bar{\mathbf{A}})UZ$. By Lemma \ref{lem:block-decomposition-n+2}, we have
\[
L_{\tilde{\mathbf{A}}}(U')=\left(s_{k}(\mathbf{A})u_1',~\left(s_{k}(\mathbf{A})-\tau s_{k-1}(\bar{\mathbf{A}})\right)u_2'\right).
\]
To simplify the notation, we denote $s_{k}(\mathbf{A})$ by $P$ and $s_{k-1}(\bar{\mathbf{A}})$ by $Q$. Set $E_{i,j}~(1\le i,j\le 2)$ be the matrix unit of $\Mat_2(\mathbb{R})$. Then
\[
(w_{\tilde{\mathbf{A}}}|_{\mathcal{U}\times\mathcal{V}})^{-1}w_{\mathbf{B}}|_{\mathcal{S}\times\mathcal{V}}(Z)=U'=P^{-1}QUZE_{1,1}+(P-\tau Q)^{-1}QUZE_{2,2}.
\]
A similar computation gives
\[
(w_{\tilde{\mathbf{A}}}|_{\mathcal{V}\times\mathcal{U}})^{-1}w_{\mathbf{B}}|_{\mathcal{S}\times\mathcal{U}}(Z)=E_{1,1}ZVQP^{-1}+E_{2,2}ZVQ(P-\tau Q)^{-1}.
\]
By \eqref{eq:w-expression}, for $Z_1,Z_2\in\mathcal{S}$ we have
\begin{align*}w(Z_1,Z_2)&=w_{\mathbf{B}}\left((w_{\tilde{\mathbf{A}}}|_{\mathcal{U}\times\mathcal{V}})^{-1}w_{\mathbf{B}}|_{\mathcal{S}\times\mathcal{V}}(Z_1),Z_2\right)+w_{\mathbf{B}}\left((w_{\tilde{\mathbf{A}}}|_{\mathcal{V}\times\mathcal{U}})^{-1}w_{\mathbf{B}}|_{\mathcal{S}\times\mathcal{U}}(Z_1),Z_2\right)\\
&=-\tr\left(\left(P^{-1}QUZ_1E_{1,1}+(P-\tau Q)^{-1}QUZ_1E_{2,2}\right)Z_2VQ\right)\\&+\tr\left(\left(E_{1,1}Z_1VQP^{-1}+E_{2,2}Z_1VQ(P-\tau Q)^{-1}\right)QUZ_2\right)\\
&=-\tr(VQP^{-1}QUZ_1E_{1,1}Z_2)-\tr\left(VQ(P-\tau Q)^{-1}QUZ_1E_{2,2}Z_2\right)\\
&+\tr(Z_2E_{1,1}Z_1VQP^{-1}QU)+\tr\left(Z_2E_{2,2}Z_1VQ(P-\tau Q)^{-1}QU\right).
\end{align*}
As $\tau\ne0$ and $\dim\Span\{s_k(\mathbf{A}),s_{k-1}(\bar{\mathbf{A}})\}=2$ by (c) of Lemma \ref{lem:specialization}, we have 
\[
\dim\Span\{QP^{-1}Q,Q(P-\tau Q)^{-1}Q\}=2.
\]
As both $QP^{-1}Q$ and $Q(P-\tau Q)^{-1}Q$ are invertible, we can choose $u\in\mathbb{R}^n$ such that
\[
\dim\Span\{QP^{-1}Qu,Q(P-\tau Q)^{-1}Qu\}=2.
\]
We further choose $v_1,v_2\in\mathbb{R}^n$ such that 
\[
(v_1,v_2)^{\mathsf{T}}(QP^{-1}Qu,Q(P-\tau Q)^{-1}Qu)=I_2.
\]
Set $U=(u,u),~V=(v_1,v_2)^{\mathsf{T}}$. For this specialization, we have
\begin{align*}
    w(Z_1,Z_2)&=-\tr\left((E_{1,1}+E_{1,2})Z_1E_{1,1}Z_2\right)-\tr\left((E_{2,1}+E_{2,2})Z_1E_{2,2}Z_2\right)\\
    &+\tr\left(Z_2E_{1,1}Z_1(E_{1,1}+E_{1,2})\right)+\tr\left(Z_2E_{2,2}Z_1(E_{2,1}+E_{2,2})\right).
\end{align*}
Under the basis $E_{1,1},~E_{1,2},~E_{2,1},~E_{2,2}$ of $\mathcal{S}$, the matrix corresponding to $w$ is
\[
\begin{pmatrix}
0&0&1&0\\0&0&0&-1\\-1&0&0&0\\0&1&0&0
\end{pmatrix},
\]
which is nonsingular. This implies that $w$ is nondegenerate.
\end{proof}

\section{Proof of the main theorem}
\begin{proof}[Proof of Theorem \ref{thm:main}]
    By Theorem \ref{thm:boundary}, Theorem \ref{thm:additional-case} and Proposition \ref{prop:n-to-n+2}, a good tuple exists for any pair $(k,n)$. By Corollary \ref{nullity-lower-bound}, good tuples form a Zariski-open subset of $\Mat_n(\mathbb{R})^k$.
\end{proof}

\begin{remark}
    One checks that the whole proof holds for any field of characteristic zero. Combining Theorem \ref{thm:main} with \cite[Theorem 1.2]{BrassilReichstein}, we obtain Theorem \ref{thm:complete}, which is a complete answer for Conjecture \ref{conj}.

\end{remark}

\section*{Acknowledgements}
The first author thanks Professor Zhang Lei and the National University of Singapore for their hospitality during her research visit. The second author thanks the Morningside Center of Mathematics for its hospitality during the 2026 Summer School on Algebra and Number Theory.

The main ideas and the proof strategies were developed by the authors independently, though artificial intelligence tools were used in some of the proofs presented in this paper. For example, the construction of $q$-forests in Subsection \ref{subsec} was first given by AI. The explicit good tuple for $(k,n+1)=(3,4)$ in the proof of Theorem \ref{thm:additional-case} is also given by AI. All content generated by AI has been completely rewritten and carefully checked. The authors take full responsibility for their validity.

\end{document}